\documentclass[10pt, reqno]{amsart}
\usepackage{graphicx, amssymb, amsmath, amsthm}
\numberwithin{equation}{section}

\usepackage{color}
\usepackage{epsfig}
\usepackage{subfigure}
\usepackage{tikz}
\usepackage[backref=page]{hyperref}

\hypersetup{urlcolor=blue, citecolor=red}

\usepackage{hyperref}

\newcommand{\R}{\mathbb{R}}

\newtheorem{theorem}{Theorem}[section]

\newtheorem{lemma}[theorem]{Lemma}

\newtheorem{proposition}[theorem]{Proposition}

\theoremstyle{definition}

\newtheorem{remark}[theorem]{Remark}

\makeatletter
\newcommand{\Extend}[5]{\ext@arrow0099{\arrowfill@#1#2#3}{#4}{#5}}
\makeatother

\begin{document}
\title[dispersive estimates for fourth order operators]{Dispersive estimates of fourth order Schr\"odinger operators with scaling-critical magnetic potentials in dimension two}

\author{Haoran Wang}
\address{Department of Mathematics, Beijing Institute of Technology, Beijing 100081, China;}
\email{wanghaoran@bit.edu.cn}

\begin{abstract}
Dispersive estimate for the fourth order Schr\"odinger operator with a class of scaling-critical magnetic potentials in dimension two was obtained by the construction of the corresponding resolvent kernel and the stationary phase method.
\end{abstract}

\maketitle

\begin{center}
 \begin{minipage}{120mm}
 \end{minipage}
 \end{center}

\section{Introduction}

Consider the fourth order Schr\"odinger operator with a scaling-critical magnetic potential on $\R^2$
\begin{equation}\label{op}
\mathcal{L}_{\mathbf{A},4}=\left(-i\nabla+\frac{\mathbf{A}(\hat{x})}{|x|}\right)^4,\quad \hat{x}=\frac{x}{|x|}\in\mathbb{S}^1,\quad x\in\R^2\setminus\{0\},
\end{equation}
where $\mathbf{A}\in W^{1,\infty}(\mathbb{S}^1;\R^2)$ is a vector function fulfilling the transversality condition
\begin{equation}\label{cond:trans}
\mathbf{A}(\hat{x})\cdot\hat{x}\equiv0,\quad \forall\hat{x}\in\mathbb{S}^1.
\end{equation}
A typical example of $\mathbf{A}$ is given by the Aharonov-Bohm potential
\begin{equation*}
\mathbf{A}(\hat{x})=\alpha|x|^{-1}(-x_2,x_1),\quad x=(x_1,x_2)\in\R^2\setminus\{0\},\quad \alpha\in\R\setminus\mathbb{Z}.
\end{equation*}
If we write $\mathbf{A}(\hat{x})=(A_1,A_2)(\hat{x})$, then the usual second order operator $\mathcal{L}_{\mathbf{A},2}=(-i\nabla+\mathbf{A}/|x|)^2$ represents the quantum Hamiltonian of a non-relativistic electron interacting with the singular magnetic field $B(x)$
\begin{equation*}
B(x)=\frac{\partial}{\partial x_1}\frac{A_2(\hat{x})}{|x|}-\frac{\partial}{\partial x_2}\frac{A_1(\hat{x})}{|x|},\quad x\in\R^2.
\end{equation*}
Since the potential $\frac{\mathbf{A}(\hat{x})}{|x|}$ is singular at zero, we should extend the operator $\mathcal{L}_{\mathbf{A},2}$ to a self-adjoint operator on $L^2(\R^2)$ so that the Schr\"odinger semigroup $e^{-it\mathcal{L}_{\mathbf{A},4}}$ generated by $\mathcal{L}_{\mathbf{A},4}$ (in the spectral sense) is well-defined. To be specific, we regard $\mathcal{L}_{\mathbf{A},2}$ as the Friedrichs extension of the symmetric operator $(-i\nabla+\mathbf{A}/|x|)^2$ (initially defined on $C_c^\infty(\R^2\setminus\{0\})$) throughout the whole paper. Hence, the spectrum of the Hamiltonian $\mathcal{L}_{\mathbf{A},4}$ consists only of the absolutely continuous part $[0,+\infty)$.
If $\mathbf{A}\equiv\mathbf{0}$, then $\mathcal{L}_{\mathbf{A},4}$ reduces to the free fourth order Schr\"odinger operator $\mathcal{L}_{\mathbf{0},4}=\Delta^2$. 
By the Fourier transform, one can easily derive the sharp pointwise time-decay estimate for any $x,y\in\R^d$ (see e.g. \cite{BKS00})
\begin{equation}\label{dec:point}
|e^{-it\Delta^2}(x,y)|\lesssim(1+|t|^{-\frac{1}{4}}|x-y|)^{-\frac{d}{3}}|t|^{-\frac{d}{4}},t\neq0,
\end{equation}
which, together with Young's inequality, immediately yields the dispersive estimate for $e^{-it\Delta^2}$
\begin{equation}\label{dis:free}
 \|e^{-it\Delta^2}\|_{L^1(\R^d)\rightarrow L^\infty(\R^d)}\lesssim|t|^{-\frac{d}{4}},\quad t\neq0.
\end{equation}
The dispersive estimates for the second order Schr\"odinger operator (with the general form $(-i\nabla+\vec{A}(x))^2+V(x)$) have been extensively studied in all dimensions (see e.g. \cite{BG12,BS19,BS20,EGG14,FFFP13,FFFP15,GYZZ22,FZZ22,MSZ23,Sch21} and the references therein). For the higher order cases, however, the picture becomes rather unclear and there are a lot of situations to be considered. In fact, the investigations concerning the higher order Schr\"odinger operators are absolutely hot topics in harmonic analysis, at least in the last few years. In particular, several authors obtained various dispersive estimates for the higher order operators $(-\Delta)^m+V$ with a decaying potential $V$ in different dimensions recently. For example, the dispersive estimate \eqref{dis:free} for $\Delta^2+V$ was obtained via the Wiener inversion technique by Hill \cite{Hil20} in dimension one. Erdo\u{g}an-Goldberg-Green \cite{EGG23} proved the dispersive bounds for $(-\Delta)^m+V$ with $m>1$ in higher dimensions $d\in(2m,4m)$ via the same method. For the fourth order operator $\Delta^2+V$, Feng-Soffer-Yao \cite{FSY18} established the Jensen-Kato bound $(1+|t|)^{-\frac{d}{4}}$ in dimensions $d\geq5$ and the local dispersive bound $O(|t|^{-\frac{1}{2}})$ in dimension three by the asymptotic expansion of the perturbed resolvent $R_V(z)=(\Delta^2+V-z)^{-1}$ at the zero threshold provided that zero is a regular point. Based on the high energy estimates for the perturbed resolvent $R_V(z)$, the authors of \cite{SWY22,LSY23,EGT19,GT19} obtained the global dispersive estimates for $\Delta^2+V$ in dimensions one, two, three and four, respectively, permitting the presence of a zero resonance or eigenvalue. Feng-Soffer-Wu-Yao \cite{FSWY20} proved local dispersive estimates for $(-\Delta)^m+V$ in dimensions $d>2m$ with a decaying potential $|V(x)|\lesssim(1+|x|)^{-\beta}$ for some $\beta>d$. Global dispersive estimate was obtained as a direct consequence of the $L^p$-boundedness of the associated wave operators by Erdo\u{g}an-Green \cite{EG22} for $\beta>d+3$ in odd dimension $d$ and \cite{EG23} for $\beta>d+4$ in even dimension $d$. It turns out that certain measure of smoothness is indispensable to guarantee the validity of the global dispersive bound for $(-\Delta)^m+V$ in high dimensions $d\geq4m$ (see \cite{EG22,EGG23}). Very recently, in a preprint \cite{EGG24}, Erdo\u{g}an-Goldberg-Green established the dispersive estimate \eqref{dis:free} for $(-\Delta)^m+V(1<m\in\mathbb{N})$ with scaling-critical potentials $V$ in high dimensions $d\in(2m,4m)$ again via the Wiener inversion theorem. For the second order operator $-\Delta+V$ in dimension two, Erdo\u{g}an-Green obtained a weighted dispersive bound in \cite{EG13a} and dispersive estimates with obstructions at zero energy in \cite{EG13b}; Erdo\u{g}an-Goldberg-Green established the $L^p$-boundedness of the associated wave operators with threshold obstructions in \cite{EGG18} (for the regular case, see Schlag's earlier paper \cite{Sch05}). Much more recently, Li-Soffer-Yao \cite{LSY23} proved the dispersive estimate \eqref{dis:free} for $\Delta^2+V$ with a decaying regular potential $V$ in dimension two. For more considerations of fourth order operator $\Delta^2+V$ in dimensions one and three,
we refer the interested readers to e.g. \cite{Hil20,MWY22,MWY23,MWY24,SWY22}. 

In this paper, we show that the Schr\"odinger propagator $e^{-it\mathcal{L}_{\mathbf{A},4}}$ satisfies the same dispersive estimate as \eqref{dis:free} for a class of scaling-critical magnetic potentials in dimension two. On the one hand, the two-dimensional situation is quite special in the sense that $\R^2$ is the lowest dimensional Euclidean space for which the addition of a magnetic potential is nontrivial since the free operator $\mathcal{L}_{\mathbf{0},4}$ is very sensitive to small perturbations. For example, the spectrum of the purely magnetic operator $(-i\nabla+\vec{A}(x))^{2}$ can be quite general, ranging from the purely absolutely continuous spectrum $[0,\infty)$ for the compactly supported magnetic field $B(x)$ (e.g. the Aharonov-Bohm solenoid), through the discrete set of infinitely degenerated eigenvalues (Landau levels) for the constant field $B(x)\equiv B\neq0$, to purely discrete spectrum for $|B(x)|\rightarrow\infty,|x|\rightarrow\infty$ (magnetic bottles); we refer the interested reader to \cite{FKV18} on how to eliminate these eigenvalues. For the Schr\"odinger operator $-\Delta+V$, Kato \cite{Kat59} established the absence of positive eigenvalues for a bounded decaying potential $V=o(|x|^{-1})$ as $|x|\rightarrow+\infty$. For more about the absence of positive eigenvalues of Schr\"odinger operators, we refer to \cite{FHHH82,IJ03,KT06,Sim69}. On the other hand, due to the degeneracy of the higher order free operator $\Delta^{m}(m\geq1)$ at zero threshold and in lower even dimension, it turns out that the asymptotic expansions of the perturbed resolvent $R_V(z)$ at zero are more complicated than the usual Schr\"odinger operator $-\Delta+V$ in dimension two (see \cite{JN01,LSY23}). For instance, there exists a compactly supported smooth potential $V$ such that $\Delta^2+V$ admits a positive eigenvalue and if the potential $V$ is bounded and repulsive (i.e. satisfying $x\cdot\nabla\leq0$), then $\Delta^2+V$ does not admit positive eigenvalues (see e.g. \cite{FSWY20}). Moreover, for a general self-adjoint operator $H$ on $L^2(\R^d)$, Costin-Soffer \cite{CS01} proved that $H+\epsilon V$ can kick off the eigenvalue located in a small interval for a generic small potential $\epsilon V$ even if $H$ only admits a simple embedded eigenvalue. Fortunately, the Schr\"odinger operator with a scaling-critical magnetic potential has neither negative nor embedded eigenvalues and the resolvent of $\mathcal{L}_{\mathbf{A},2}$ is regular at zero, although its resolvent kernel is a little more complicated than the free case (see e.g. \cite{FZZ23,GYZZ22}). 
Throughout the whole paper, it will be convenient to write $X\lesssim Y$ if there exists some constant $C>0$ such that $X\leq CY$. Now we state the main result.

\begin{theorem}\label{thm:dis}
Let $\mathcal{L}_{\mathbf{A},4}$ be the self-adjoint Schr\"odinger operator with a scaling-critical magnetic potential given by \eqref{op}, then we have the dispersive estimate
\begin{equation}\label{dis:S}
\|e^{-it\mathcal{L}_{\mathbf{A},4}}\|_{L^1(\R^2)\rightarrow L^\infty(\R^2)}\lesssim|t|^{-\frac{1}{2}},\quad t\neq0.
\end{equation}
\end{theorem}
\begin{remark}\label{rem:1.2}
Theorem \ref{thm:dis} is a natural generalization of the dispersive estimate \eqref{dis:free} for the free operator $\Delta^2$. We do not include an additional rough potential $V$ (i.e.  $\mathcal{L}_{\mathbf{A},4}+V$ with $|V(x)|\lesssim(1+|x|)^{-\beta}$) here since once we have proved the dispersive estimate \eqref{dis:S} for the magnetic operator $\mathcal{L}_{\mathbf{A},4}$, then, viewing $\mathcal{L}_{\mathbf{A},4}+V$ as a perturbation of $\mathcal{L}_{\mathbf{A},4}$ (similar to the treatment of $\Delta^2+V$), one may closely follow the steps of Li-Soffer-Yao \cite{LSY23} (with some obvious modifications) to obtain the dispersive estimate for $\mathcal{L}_{\mathbf{A},4}+V$. The paper \cite{LSY23} is quite lengthy (but the idea is the classical perturbation method), so we only consider the purely magnetic case $\mathcal{L}_{\mathbf{A},4}$ to avoid repeating.
Apart from the case of an additional perturbation $V$, there are also many interesting situations to be considered. For example, one may consider the dispersive or smoothing estimates for the general higher order Schr\"odinger operators (with electromagnetic potentials) in two and higher dimensions.
\end{remark}

\section{preliminaries}

In this section, we obtain the resolvent kernel of the operator $\mathcal{L}_{\mathbf{A},4}$ by using the second order resolvent kernel constructed in \cite{GYZZ22} (see also \cite{FZZ23}).

\subsection{The magnetic flux}

If we write $\hat{x}=(\cos\theta,\sin\theta)$ for $\hat{x}\in\mathbb{S}^1$, then we have $\partial_\theta=\hat{x}_1\partial_{\hat{x}_2}-\hat{x}_2\partial_{\hat{x}_1}$ and $\partial_\theta^2=\Delta_{\mathbb{S}^1}$. Set $x=r\hat{x}$, then the operator $\mathcal{L}_{\mathbf{A},1}$ (i.e. \eqref{op} with $m=1$) can be rewritten as
\begin{equation*}
\mathcal{L}_{\mathbf{A},1}=-\partial_r^2-r^{-1}\partial_r+r^{-2}L_\mathbf{A},
\end{equation*}
where
\begin{equation*}
L_\mathbf{A}=(-i\nabla_{\mathbb{S}^1}+\mathbf{A}(\hat{x}))^2
=-\Delta_{\mathbb{S}^1}+|\mathbf{A}(\hat{x})|^2+i\mathrm{div}_{\mathbb{S}^1}\mathbf{A}(\hat{x})+2i\mathbf{A}(\hat{x})\cdot\nabla_{\mathbb{S}^1}.
\end{equation*}
Define $\alpha(\theta):[0,2\pi]\rightarrow\R$ by 
\begin{equation*}
\alpha(\theta)=\mathbf{A}(\cos\theta,\sin\theta)\cdot(-\sin\theta,\cos\theta),
\end{equation*}
then, in view of the transversality assumption \eqref{cond:trans}, we have
\begin{equation*}
\mathbf{A}(\cos\theta,\sin\theta)=(-\alpha(\theta)\sin\theta,\alpha(\theta)\cos\theta),\quad \theta\in[0,2\pi]
\end{equation*}
and thus
\begin{equation*}
\begin{split}
L_\mathbf{A}&=-\Delta_{\mathbb{S}^1}+|\mathbf{A}(\hat{x})|^2+i\mathrm{div}_{\mathbb{S}^1}\mathbf{A}(\hat{x})+2i\mathbf{A}(\hat{x})\cdot\nabla_{\mathbb{S}^1}\\
&=-\partial_\theta^2+|\alpha(\theta)|^2+i\alpha'(\theta)+2i\alpha(\theta)\partial_\theta\\
&=(-i\partial_\theta+\alpha(\theta))^2.
\end{split}
\end{equation*}
In what follows, the total magnetic flux $\alpha$ is defined by
\begin{equation}\label{alpha}
\alpha:=\frac{1}{2\pi}\int_0^{2\pi}\alpha(\theta)d\theta.
\end{equation}

\subsection{Resolvent kernel for $\mathcal{L}_{\mathbf{A},4}$}

To obtain the kernel of the resolvent $(\mathcal{L}_{\mathbf{A},4}-z)^{-1}$, we need the resolvent kernel of the second order operator $\mathcal{L}_{\mathbf{A},2}$ (see \cite{FZZ23,GYZZ22}). The resolvent $(\mathcal{L}_{\mathbf{A},4}-z)^{-1}$ is defined for $\lambda>0$ by
\begin{equation*}
R_4^\pm(\lambda^4):=(\mathcal{L}_{\mathbf{A},4}-(\lambda^{4}\pm i0))^{-1}=\lim_{\epsilon\searrow0}(\mathcal{L}_{\mathbf{A},4}-(\lambda^{4}\pm i\epsilon))^{-1}.
\end{equation*}
\begin{proposition}[Resolvent kernel of $\mathcal{L}_{\mathbf{A},2}$]\label{prop:res-2}
Let $\theta_x,\theta_y$ denote the angles between the positive real axis and $x,y\in\R^2$ respectively and define the partial magnetic flux $\alpha_{12}$ by
\begin{equation*}
\alpha_{12}:=\alpha(\theta_1,\theta_2)=\int_{\theta_1}^{\theta_2}\alpha(\vartheta)d\vartheta.
\end{equation*} 
The Hankel function $H_0^\pm(z)$ is defined for $z\in\mathbb{C}$ by
\begin{equation*}
H_0^\pm(z)=J_0(z)\pm iY_0(z),
\end{equation*}
where $J_0(z)$ and $Y_0(z)$ are the Bessel functions of the first kind and the second kind, respectively. For $x,y\in\R^2$ and $s>0$, we define the vector $\mathbf{n}=(|x|+|y|,\sqrt{2|x||y|(\cosh s-1)})$, then the kernel of the resolvent $(\mathcal{L}_{\mathbf{A},2}-(\lambda^2\pm i0))^{-1}$ is given by
\begin{equation*}
\begin{split}
R_2^\pm(\lambda^2)(x,y):&=(\mathcal{L}_{\mathbf{A},2}-(\lambda^2\pm i0))^{-1}(x,y)\\
&=\pm\frac{i}{4\pi}\left(H_0^\pm(\lambda|x-y|)A_\alpha(\theta_x,\theta_y)+\int_0^\infty H_0^\pm(\lambda|\mathbf{n}|)B_\alpha(s,\theta_x,\theta_y)ds\right)
\end{split}
\end{equation*}
with 
\begin{equation*}
A_\alpha(\theta_x,\theta_y)=\frac{e^{i\alpha_{12}}}{4\pi^2}\left(1_{[0,\pi]}(|\theta_x-\theta_y|)+e^{-2\pi i\alpha}1_{[\pi,2\pi]}(|\theta_x-\theta_y|)\right)
\end{equation*}
and
\begin{align*}
&B_\alpha(s,\theta_x,\theta_y)=-\frac{e^{-i\alpha(\theta_x-\theta_y)-i\alpha_{12}}}{4\pi^2}\bigg(\sin(\pi|\alpha|)e^{-s|\alpha|}\\
&+\sin(\alpha\pi)\frac{(e^{-s}+\cos(\theta_x-\theta_y))\sinh(s\alpha)+i\sin(\theta_x-\theta_y)\cosh(s\alpha)}{\cosh s+\cos(\theta_x-\theta_y)}\bigg).
\end{align*}
Here $\alpha$ is the constant defined in \eqref{alpha}.
\end{proposition}
Similar to the resolvent identity for the free operator $\Delta^2$ 
\begin{equation*}
\left(\Delta^2-z\right)^{-1}=\frac{1}{2\sqrt{z}}\left((-\Delta-\sqrt{z})^{-1}-(-\Delta+\sqrt{z})^{-1}\right),\quad z\in\mathbb{C},
\end{equation*}
one obtains the kernel of the resolvent $R_4^\pm(\lambda^{4})=(\mathcal{L}_{\mathbf{A},4}-(\lambda^{4}\pm i0))^{-1}$.
\begin{proposition}\label{prop:res-m}
Let $R_4^\pm(\lambda^{4})(x,y)$ be the kernel function of the resolvent $(\mathcal{L}_{\mathbf{A},4}-(\lambda^{4}\pm i0))^{-1}$, then, using the notations of Proposition \ref{prop:res-2}, we have
\begin{equation*}
\begin{split}
R_4^\pm(\lambda^{4})(x,y)
=&\pm\frac{iA_\alpha(\theta_x,\theta_y)}{8\pi\lambda^{2}}\left(H_0^\pm(\lambda|x-y|)-H_0^\pm(i\lambda|x-y|)\right)\\
&+\frac{\pm i}{8\pi\lambda^{2}}\int_0^\infty\left(H_0^\pm(\lambda|\mathbf{n}(s)|)-H_0^\pm(i\lambda|\mathbf{n}(s)|)\right)B_\alpha(s,\theta_x,\theta_y)ds.
\end{split}
\end{equation*}
\end{proposition}
To analysis the resolvent kernel $R_4^\pm(\lambda^{4})(x,y)$, we need the asymptotic behaviors of the Hankel function $H_0^\pm(z)$ (c.f. \cite{AS65}).
\begin{lemma}\label{lem:free}
The Hankel function $H_0^\pm$ is defined by
\begin{equation*}
H_0^\pm(z)=J_0(z)\pm iY_0(z),
\end{equation*}
where 
\begin{equation*}
\begin{split}
J_0(z)&=\sum_{k=0}^\infty(-1)^k\frac{z^{2k}}{2^{2k}(k!)^2},\quad z\rightarrow0,\\
Y_0(z)&=\frac{2}{\pi}(\ln(z/2)+\gamma)J_0(z)+\frac{2}{\pi}\sum_{k=1}^\infty(-1)^{k-1}\left(\sum_{j=1}^k\frac{1}{j}\right)\frac{z^{2k}}{2^{2k}(k!)^2},z\rightarrow0.
\end{split}
\end{equation*}
Here $\gamma$ is the Euler constant. For $|z|>1$, one has
\begin{equation}\label{rep:hankel}
H_0^\pm(z)=e^{\pm iz}\omega_\pm(z),\quad |\omega_\pm^{(k)}(z)|\lesssim(1+|z|)^{-\frac{1}{2}-k},\quad\forall k\geq0.
\end{equation}
\end{lemma}
We also need the Van der Corput lemma from \cite[P. 334]{Ste93}.
\begin{lemma}\label{lem:vdc}
Let $\phi,\psi\in C_0^\infty(a,b)$ with $|\phi^{(k)}(x)|\geq1$ for all $x\in(a,b)$. If $k\geq2$ or $k=1$ and $\phi'(x)$ is monotonic, then we have
\begin{equation*}
\left|\int_a^be^{i\lambda\phi(x)}\psi(x)dx\right|\lesssim_k\lambda^{-\frac{1}{k}}\left(|\psi(b)|+\int_a^b|\psi'(x)|dx\right).
\end{equation*}
\end{lemma}

\section{proof of Theorem \ref{thm:dis}}

In this section, we prove the dispersive estimate \eqref{dis:S} for the magnetic propagator $e^{-it\mathcal{L}_{\mathbf{A},4}}$ by using the resolvent kernel in Proposition \ref{prop:res-m} and Lemma \ref{lem:free}.

By the well-known Stone formula, we have
\begin{equation}\label{stone}
e^{-it\mathcal{L}_{\mathbf{A},4}}f=\frac{2}{\pi i}\int_0^\infty\lambda^{3}e^{-it\lambda^{4}}\left(R_4^+(\lambda^{4})-R_4^-(\lambda^{4})\right)fd\lambda.
\end{equation}
By Proposition \ref{prop:res-m}, we get the kernel of $e^{-it\mathcal{L}_{\mathbf{A},4}}$
\begin{equation*}
\begin{split}
&e^{-it\mathcal{L}_{\mathbf{A},4}}(x,y)\\
&=\frac{A_\alpha(\theta_x,\theta_y)}{4\pi^2}\sum_\pm\int_0^\infty\lambda e^{-it\lambda^{4}}(H_0^\pm(\lambda|x-y|)-H_0^\pm(i\lambda|x-y|))d\lambda\\
&+\frac{1}{4\pi^2}\sum_\pm\int_0^\infty\lambda e^{-it\lambda^{4}}\int_0^\infty\left(H_0^\pm(\lambda|\mathbf{n}(s)|)-H_0^\pm(i\lambda|\mathbf{n}(s)|)\right)B_\alpha(s,\theta_x,\theta_y)dsd\lambda.
\end{split}
\end{equation*}
Let $F^\pm(\rho):=H_0^\pm(\rho)-H_0^\pm(i\rho)$, then, in view of Lemma \ref{lem:free}, we have
\begin{equation}\label{F}
\begin{split}
F^\pm(\rho)=&\chi(\rho)(c_1^\pm+c_2^\pm\rho^2+c_3^\pm\rho^2\ln\rho+c_4^\pm\rho^4+c_5^\pm\rho^6+c_6^\pm\rho^6\ln\rho+\cdots)\\
&+i(1-\chi(\rho))(c_7^\pm e^{\pm i\rho}\omega_\pm(\rho)-c_8^\pm e^{-\rho}\omega_+(i\rho)),
\end{split}
\end{equation}
where $c_j^\pm,j=1,2,\ldots,8$ are suitable constants and $\chi$ is a smooth and compactly supported function (i.e. $\chi\in C^\infty(\R)$ with $\eta(s)=1$ for $|s|\leq1/2$ and $\chi(s)=0$ for $|s|\geq1$).
Let 
\begin{equation*}
G^\pm(\lambda|x-y|):=\int_0^\infty\left(H_0^\pm(\lambda|\mathbf{n}(s)|)-H_0^\pm(i\lambda|\mathbf{n}(s)|)\right)B_\alpha(s,\theta_x,\theta_y)ds,
\end{equation*}
then, from Lemma \ref{lem:free} again, we obtain
\begin{equation}\label{G}
\begin{split}
G^\pm(\lambda|x-y|)&=\int_0^\infty\bigg[\chi(\lambda|\mathbf{n}(s)|)(c_1^\pm+c_2^\pm(\lambda|\mathbf{n}(s)|)^2\ln(\lambda|\mathbf{n}(s)|)\\
&+c_3^\pm(\lambda|\mathbf{n}(s)|)^2+c_4^\pm(\lambda|\mathbf{n}(s)|)^4+c_5^\pm(\lambda|\mathbf{n}(s)|)^6\ln(\lambda|\mathbf{n}(s)|)+\cdots)\\
&+c_6^\pm(\lambda|\mathbf{n}(s)|)^6+i(1-\chi(\lambda|\mathbf{n}(s)|))\Big(c_7^\pm e^{\pm i(\lambda|\mathbf{n}(s)|)}\omega_\pm(\lambda|\mathbf{n}(s)|)\\
&-c_8^\pm e^{-\lambda|\mathbf{n}(s)|}\omega_+(i\lambda|\mathbf{n}(s)|)\Big)\bigg]B_\alpha(s,\theta_x,\theta_y)ds,
\end{split}
\end{equation}
where $c_j^\pm,j=1,2,\ldots,8,\chi$ are the same as in \eqref{F} and $\mathbf{n}(s)$ is defined in Proposition \ref{prop:res-2}.

For convenience, we denote
\begin{equation}\label{K-1}
K_1^\pm(x,y)=\int_0^\infty\lambda e^{-it\lambda^{4}}F^\pm(\lambda|x-y|)d\lambda
\end{equation}
and
\begin{equation}\label{K-2}
K_2^\pm(x,y)=\int_0^\infty\lambda e^{-it\lambda^{4}}G^\pm(\lambda|x-y|)d\lambda.
\end{equation}
Now the proof of the dispersive estimate
\begin{equation*}
\|e^{-it\mathcal{L}_{\mathbf{A},4}}\|_{L^1\rightarrow L^\infty}\lesssim|t|^{-\frac{1}{2}}
\end{equation*}
is obviously reduced to proving
\begin{equation}\label{bd-K}
\sup_{x,y\in\R^2}|K_j^\pm(x,y)|\lesssim|t|^{-\frac{1}{2}},\quad j=1,2.
\end{equation}
To apply Lemma \ref{lem:vdc}, we take $\varphi_j(s):=\chi(2^{-j}s)-\chi(2^{-j+1}s)$ so that we have $\varphi_j(s)=\varphi_0(2^{-j}s)$ with supp$\varphi_0\subset[\frac{1}{4},1]$ and $\sum_{j\in\mathbb{Z}}\varphi_0(2^{-j}s)=1$ for $s\neq0$.
Let 
\begin{equation}\label{K-1j}
K_{1,j}^\pm(x,y)=\int_0^\infty\lambda e^{-it\lambda^{4}}\varphi_0(2^{-j}\lambda)F^\pm(\lambda|x-y|)d\lambda,
\end{equation}
then we have
\begin{equation*}
  K_1^\pm(x,y)=\sum_{j\in\mathbb{Z}}K_{1,j}^\pm(x,y).
\end{equation*}
Similarly, let
\begin{equation}\label{K-2j}
K_{2,j}^\pm(x,y)=\int_0^\infty\lambda\varphi_0(2^{-j}\lambda)e^{-it\lambda^{4}}G^\pm(\lambda|x-y|)d\lambda
\end{equation}
so that we have
\begin{equation*}
K_2^\pm(x,y)=\sum_{j\in\mathbb{Z}}K_{2,j}^\pm(x,y).
\end{equation*}
It is trivial to verify that $|A_\alpha(\theta_x,\theta_y)|\leq1$ uniformly (see Proposition \ref{prop:res-2}), hence the proof of the bound \eqref{bd-K} is done if we could show
\begin{lemma}\label{lem:reduce}
Let $K_{1,j}^\pm$ and $K_{2,j}^\pm$ be given by \eqref{K-1j} and \eqref{K-2j} respectively, then there exists some $j_0\in\mathbb{Z}$ depending on $t,x,y$ such that for all $j\in\mathbb{Z},\ell=1,2$ and $t\neq0,x\neq y$, we have
\begin{equation}\label{reduce}
|K_{\ell,j}^\pm(x,y)|\lesssim2^{2j}\times
\begin{cases}
(1+2^{4j}|t|)^{-\frac{1}{2}},&|j-j_0|\leq2,\\
(1+2^{4j}|t|)^{-1},&|j-j_0|>2.
\end{cases}
\end{equation}
\end{lemma}
\begin{proof}
We assume $t>0$. We first prove \eqref{reduce} for $\ell=1$. 
Let $\lambda=2^j\sigma$, then we have
\begin{align*}
K_{1,j}^\pm(t,x,y)&=\int_0^\infty\lambda e^{-it\lambda^{4}}\varphi_j(\lambda)F^\pm(\lambda|x-y|)d\lambda\\
&=2^{2j}\int_0^\infty\sigma e^{-it2^{4j}\sigma^{4}}\varphi_0(\sigma)F^\pm(2^j\sigma|x-y|)d\sigma.
\end{align*}
Let $\tilde{F}^\pm(\lambda|x-y|)=e^{\mp i\lambda|x-y|}F^\pm(\lambda|x-y|)$, then we have
\begin{equation}\label{K-1j:scaling}
\begin{split}
K_{1,j}^\pm(t,x,y)=&2^{2j}\int_0^\infty\sigma e^{-it2^{4j}\sigma^{4}\pm i2^j\sigma|x-y|}\\
&\quad \times\varphi_0(\sigma)\tilde{F}^\pm(2^j\sigma|x-y|)d\sigma.
\end{split}
\end{equation}
To prove \eqref{K-1j:scaling}, we consider two cases according to $|x-y|\leq2^{-j}$ and $|x-y|\geq2^{-j}$.

{\bf Case 1.} For $|x-y|\leq2^{-j}$, we further consider two subcases according to $|t|\leq2^{-4j}$ and $|t|\geq2^{-4j}$. If $|t|\leq2^{-4j}$, then, since the support of $\varphi_0$ is contained in $[1/4,1]$, we have
\begin{equation*}
|K_{1,j}^\pm(t,x,y)|\lesssim2^{2j}\int_{1/4}^1|\sigma\tilde{F}^\pm(2^j\sigma|x-y|)|d\sigma\lesssim2^{2j}(1+2^{4j}|t|)^{-1}.
\end{equation*}
If $|t|\geq2^{-4j}$, then, by integration by parts, we have
\begin{equation*}
K_{1,j}^\pm(t,x,y)=\frac{2^{2j}}{4it2^{4j}}\int_{1/4}^1e^{-it2^{4j}\sigma^{4}}\partial_s\left(\sigma^{-2}e^{\pm i2^j\sigma|x-y|}\varphi_0(\sigma)\tilde{F}^\pm(2^j\sigma|x-y|)\right)d\sigma.
\end{equation*}
Since the support of $\varphi_0$ is contained in $[1/4,1]$, it follows from \eqref{F} and \eqref{rep:hankel} that
\begin{equation*}
|\tilde{F}^\pm(2^j\sigma|x-y|)|+|\partial_\sigma\tilde{F}^\pm(2^j\sigma|x-y|)|\lesssim1, \quad\forall x,y\in\R^2.
\end{equation*}
Hence, by Lemma \ref{lem:vdc} (Van der Corput's lemma), we obtain
\begin{equation}\label{bd-K1j}
|K_{1,j}^\pm(t,x,y)|\lesssim2^{2j}(1+2^{4j}|t|)^{-1},\quad \forall x,y\in\R^2.
\end{equation}

{\bf Case 2.} For $|x-y|\geq2^{-j}$, we consider $K_{1,j}^-(t,x,y)$ and $K_{1,j}^+(t,x,y)$ separately. For $K_{1,j}^-(t,x,y)$, if $|t|\leq2^{-4j}$, then we have 
\begin{equation*}
|K_{1,j}^-(t,x,y)|\lesssim2^{2j}(1+2^{4j}|t|)^{-1},\quad\forall x,y\in\R^2.
\end{equation*}
If $|t|\geq 2^{-4j}$, then we write
\begin{equation*}
K_{1,j}^-(t,x,y)=2^{2j}\int_0^\infty\sigma e^{-i(t2^{4j}\sigma^{4}+2^j\sigma|x-y|)}\varphi_0(\sigma)\tilde{F}^-(2^j\sigma|x-y|)d\sigma.
\end{equation*}
Let $\phi_\pm(\sigma)=t2^{4j}\sigma^{4}\pm2^j\sigma|x-y|$, then we have (remember that supp$\varphi_0(\sigma)\subset[1/4,1]$)
\begin{equation*}
|\partial_\sigma\phi_+(\sigma)|=|4t2^{4j}\sigma^{3}+2^j|x-y||\gtrsim1+2^{4j}|t|,
\end{equation*}
which implies that the phase function $\phi_+(\sigma)$ is regular in $\sigma\in[1/4,1]$. By Lemma \ref{lem:vdc}, we have
\begin{equation*}
|K_{1,j}^-(t,x,y)|\lesssim2^{2j}(1+2^{4j}|t|)^{-1},\quad\forall x,y\in\R^2.
\end{equation*}
For $K_{1,j}^+(t,x,y)$, if $|t|\leq2^{-4j}$, then, similarly, we have 
\begin{equation*}
|K_{1,j}^+(t,x,y)|\lesssim2^{2j}(1+2^{4j}|t|)^{-1},\quad\forall x,y\in\R^2.
\end{equation*}
If $|t|\geq2^{-4j}$, then we compute $\partial_\sigma\phi_-(\sigma)=4t2^{4j}\sigma^{3}-2^j|x-y|$, which implies that the phase function $\phi_-(\sigma)$ of the oscillatory integral
\begin{equation*}
K_{1,j}^+(t,x,y)=2^{2j}\int_0^\infty\sigma e^{-i(t2^{4j}\sigma^{4}-2^j\sigma|x-y|)}\varphi_0(\sigma)\tilde{F}^+(2^j\sigma|x-y|)d\sigma
\end{equation*}
admits a critical point $\sigma_0$ in $[1/4,1]$ (i.e. $\sigma_0=2^{-j-\frac{2}{3}}\sqrt[3]{\frac{|x-y|}{t}}$).
Let $j_0=[\frac{1}{3}\log_2\frac{|x-y|}{|t|}]$, then we have $-2+j<\frac{1}{3}\log_2\frac{\Phi(z)}{t}<3+j$ when $|j-j_0|<2$, which implies that such a critical point $\sigma_0$ does exist. 

One can check that for $\sigma\in[1/4,1]$
\begin{equation*}
\partial_\sigma^2\phi_-(\sigma)=12t2^{4j}\sigma^2\gtrsim2^{4j}|t|,
\end{equation*}
so it follows by Lemma \ref{lem:vdc} that
\begin{equation*}
|K_{1,j}^+(t,x,y)|\lesssim2^{2j}(2^{4j}|t|)^{-\frac{1}{2}}\left(\varphi_0(1)\tilde{F}^+(2^j|x-y|)+\int_{1/4}^1\left|\partial_\sigma\left(\sigma\varphi_0(\sigma)\tilde{F}^+(2^j\sigma|x-y|)\right)\right|ds\right).
\end{equation*}
Since $\varphi_0(1)=0$ and 
\begin{equation*}
|\tilde{F}^+(2^j\sigma|x-y|)|+|\partial_\sigma\tilde{F}^+(2^j\sigma|x-y|)|\lesssim1, \quad\forall x,y\in\R^2,
\end{equation*}
we obtain
\begin{equation*}
|K_{1,j}^+(t,x,y)|\lesssim2^{2j}(2^{4j}|t|)^{-\frac{1}{2}}\lesssim2^{2j}(1+2^{4j}|t|)^{-\frac{1}{2}}.
\end{equation*}
When $|j-j_0|\geq2$, the phase function $\phi_-(\sigma)$ is regular in $[1/4,1]$. By Lemma \ref{lem:vdc} again, we obtain 
\begin{equation*}
|K_{1,j}^+(t,x,y)|\lesssim2^{2j}(1+2^{4j}|t|)^{-\frac{1}{2}}.
\end{equation*}
Collecting {\bf Case 1} and {\bf Case 2}, we finish the proof of \eqref{reduce} for $\ell=1$.

It remains to verify \eqref{reduce} for $\ell=2$, i.e.
\begin{equation*}
K_{2,j}^\pm(t,x,y)=\int_0^\infty\lambda e^{-it\lambda^{4}}\varphi_j(\lambda)G^\pm(\lambda|x-y|)d\lambda,
\end{equation*}
where $G^\pm(\cdot)$ is given by \eqref{G}. The proof is similar to the case $\ell=1$ (with some modifications), we give the details for the convenience of the readers.
Let $\lambda=2^j\sigma$, then we have
\begin{equation*}
K_{2,j}^\pm(t,x,y)=2^{2j}\int_0^\infty\sigma e^{-it2^{4j}\sigma^{4}}\varphi_0(\sigma)G^\pm(2^j\sigma|x-y|)d\sigma.
\end{equation*}
Let $\tilde{G}^\pm(\lambda|x-y|)=e^{\mp i\lambda|x-y|}G^\pm(\lambda|x-y|)$, then we have
\begin{equation}\label{K-1j:scaling}
\begin{split}
K_{2,j}^\pm(t,x,y)=&2^{2j}\int_0^\infty\sigma e^{-it2^{4j}\sigma^{4}\pm i2^j\sigma|x-y|}\\
&\varphi_0(\sigma)\tilde{G}^\pm(2^j\sigma|x-y|)d\sigma.
\end{split}
\end{equation}
By the elementary analysis, there exists a finite constant $s_0\in(0,\infty)$ such that
\begin{equation*}
G^\pm(\lambda|x-y|)=\left(H_0^\pm(\lambda|\mathbf{n}(s_0)|)-H_0^\pm(i\lambda|\mathbf{n}(s_0)|)\right)\int_0^\infty B_\alpha(s,\theta_x,\theta_y)ds.
\end{equation*}
From Proposition \ref{prop:res-2}, we know (see also \cite[(3.7)]{FZZ23})
\begin{equation*}
\int_0^\infty|B_\alpha(s,\theta_x,\theta_y)|ds\lesssim1.
\end{equation*}
To prove \eqref{K-1j:scaling}, we consider two cases according to $|x-y|\leq2^{-j}$ and $|x-y|\geq2^{-j}$.

{\bf Case 1.} For $|x-y|\leq2^{-j}$, we further consider two subcases according to $|t|\leq2^{-4j}$ and $|t|\geq2^{-4j}$. If $|t|\leq2^{-4j}$, then, due to supp$\varphi_0\subset[1/4,1]$ and the fact $|\mathbf{n}(s_0)|\geq|x-y|$, we have
\begin{equation*}
|K_{2,j}^\pm(t,x,y)|\lesssim\int_{1/4}^1|\sigma G^\pm(2^j\sigma|x-y|)|d\sigma\lesssim2^{2j}(1+2^{4j}|t|)^{-1}.
\end{equation*}
If $|t|\geq2^{-4j}$, then, by integration by parts, we have
\begin{equation*}
K_{2,j}^\pm(t,x,y)=\frac{2^{2j}}{4it2^{4j}}\int_{1/4}^1e^{-it2^{4j}\sigma^{4}}\partial_\sigma\left(\sigma^{-2}e^{\pm i2^j\sigma|x-y|}\varphi_0(\sigma)\tilde{G}^\pm(2^j\sigma|x-y|)\right)d\sigma.
\end{equation*}
Noting that supp$\varphi_0\subset[1/4,1]$, it follows from \eqref{G} and \eqref{rep:hankel} that
\begin{align*}
&|\tilde{G}^\pm(2^j\sigma|x-y|)|=\bigg|e^{\mp i2^j\sigma|x-y|}\int_0^\infty\\
&\left(H_0^\pm(2^j\sigma|\mathbf{n}(s)|)-H_0^\pm(i2^j\sigma|\mathbf{n}(s)|)\right)B_\alpha(s,\theta_x,\theta_y)ds\bigg|
\end{align*}
and
\begin{equation*}
\begin{split}
&|\partial_\sigma\tilde{G}^\pm(2^j\sigma|x-y|)|\\
&=\bigg|\mp i2^j|x-y|e^{\mp i2^j\sigma|x-y|}\int_0^\infty\left(H_0^\pm(2^j\sigma|\mathbf{n}(s)|)-H_0^\pm(i2^j\sigma|\mathbf{n}(s)|)\right)B_\alpha(s,\theta_x,\theta_y)ds\\
&+e^{\mp i2^j\sigma|x-y|}\int_0^\infty\partial_\sigma\left(H_0^\pm(2^j\sigma|\mathbf{n}(s)|)-H_0^\pm(i2^j\sigma|\mathbf{n}(s)|)\right)B_\alpha(s,\theta_x,\theta_y)ds\bigg|
\end{split}
\end{equation*}
are both bounded by some constant uniformly in $x,y\in\R^2$.

Hence, by Lemma \ref{lem:vdc} (Van der Corput's lemma), we obtain
\begin{equation}\label{bd-K2j}
|K_{2,j}^\pm(t,x,y)|\lesssim2^{2j}(1+2^{4j}|t|)^{-1},\quad \forall x,y\in\R^2.
\end{equation}

{\bf Case 2.} For $|x-y|\geq2^{-j}$, we consider $K_{2,j}^-(t,x,y)$ and $K_{1,j}^+(t,x,y)$ separately. For $K_{2,j}^-(t,x,y)$, if $|t|\leq2^{-4j}$, then we have
\begin{equation*}
|K_{2,j}^-(t,x,y)|\lesssim2^{2j}(1+2^{4j}|t|)^{-1},\quad\forall x,y\in\R^2.
\end{equation*}
If $|t|\geq 2^{-4j}$, then we write
\begin{equation*}
K_{2,j}^-(t,x,y)=2^{2j}\int_0^\infty\sigma e^{-i(t2^{4j}\sigma^{4}+2^j\sigma|x-y|)}\varphi_0(\sigma)\tilde{G}^-(2^j\sigma|x-y|)d\sigma.
\end{equation*}
Recall $\phi_\pm(\sigma)=t2^{4j}\sigma^{4}\pm2^j\sigma|x-y|$, then we have (remember that supp$\varphi_0(\sigma)\subset[1/4,1]$)
\begin{equation*}
|\partial_\sigma\phi_+(\sigma)|=|4t2^{4j}\sigma^{3}+2^j|x-y||\gtrsim1+2^{4j}|t|,
\end{equation*}
which implies that the phase function $\phi_+(\sigma)$ is regular in $\sigma\in[1/4,1]$. By Lemma \ref{lem:vdc}, we have
\begin{equation*}
|K_{2,j}^-(t,x,y)|\lesssim2^{2j}(1+2^{4j}|t|)^{-1},\quad\forall x,y\in\R^2.
\end{equation*}
For $K_{2,j}^+(t,x,y)$, if $|t|\leq2^{-4j}$, then, similarly, we have
\begin{equation*}
|K_{2,j}^+(t,x,y)|\lesssim2^{2j}(1+2^{4j}|t|)^{-1},\quad\forall x,y\in\R^2.
\end{equation*}
If $|t|\geq2^{-4j}$, then we compute $\partial_\sigma\phi_-(\sigma)=4t2^{4j}\sigma^{3}-2^j|x-y|$, which implies that the phase function $\phi_-(\sigma)$ of the oscillatory integral
\begin{equation*}
K_{2,j}^+(t,x,y)=2^{2j}\int_0^\infty\sigma e^{-i(t2^{4j}\sigma^{4}-2^j\sigma|x-y|)}\varphi_0(\sigma)\tilde{G}^+(2^j\sigma|x-y|)d\sigma
\end{equation*}
admits a critical point $\sigma_0$ in $[1/4,1]$.
One can check as before that for $\sigma\in[1/4,1]$
\begin{equation*}
\partial_\sigma^2\phi_-(\sigma)=12t2^{4j}\sigma^2\gtrsim2^{4j}|t|,
\end{equation*}
so it follows by Lemma \ref{lem:vdc} that
\begin{equation*}
\begin{split}
|K_{2,j}^+(t,x,y)|\lesssim&2^{2j}(2^{4j}|t|)^{-\frac{1}{2}}\bigg(\varphi_0(1)\tilde{G}^+(2^j|x-y|)\\
&+\int_{1/4}^1\left|\partial_\sigma\left(\sigma\varphi_0(\sigma)\tilde{G}^+(2^j\sigma|x-y|)\right)\right|ds\bigg).
\end{split}
\end{equation*}
Since $\varphi_0(1)=0$ and
\begin{equation*}
|\tilde{G}^+(2^j\sigma|x-y|)|+|\partial_\sigma\tilde{G}^+(2^j\sigma|x-y|)|\lesssim1, \quad\forall x,y\in\R^2,
\end{equation*}
we obtain
\begin{equation*}
|K_{2,j}^+(t,x,y)|\lesssim2^{2j}(2^{4j}|t|)^{-\frac{1}{2}}\lesssim2^{2j}(1+2^{4j}|t|)^{-\frac{1}{2}}.
\end{equation*}
When $|j-j_0|\geq2$ (with $j_0=[\frac{1}{3}\log_2\frac{|x-y|}{|t|}]$ as before), the phase function $\phi_-(\sigma)$ is regular in $[1/4,1]$. By Lemma \ref{lem:vdc} again, we obtain
\begin{equation*}
|K_{2,j}^+(t,x,y)|\lesssim2^{2j}(1+2^{4j}|t|)^{-\frac{1}{2}}.
\end{equation*}
Therefore, we finish the proof of \eqref{reduce} for $\ell=2$.

Above all, the proof of this lemma is completed.
\end{proof}


\begin{remark}
The expected dispersive estimate \eqref{dis:S} follows immediately from Lemma \ref{lem:reduce}. Indeed, from \eqref{stone}, \eqref{K-1j},\eqref{K-2j} and Proposition \ref{prop:res-m}, we have 
\begin{equation*}
\begin{split}
K(x,y):&=e^{-it\mathcal{L}_{\mathbf{A},4}}(x,y)\\
&=\sum_{j\in\mathbb{Z}}\left(K_{1,j}^\pm(x,y)+K_{2,j}^\pm(x,y)\right).
\end{split}
\end{equation*}
Now from Lemma \ref{lem:reduce}, we obtain
\begin{equation*}
\begin{split}
|K(x,y)|&\leq\left(\sum_{|j-j_0|\leq2}+\sum_{|j-j_0|\geq2}\right)(K_{1,j}^\pm(x,y)+K_{2,j}^\pm(x,y))\\
&\lesssim\sum_{|j-j_0|\leq2}2^{2j}(1+2^{4j}|t|)^{-\frac{1}{2}}+\sum_{|j|\geq j'_0+1}2^{2j}(1+2^{4j}|t|)^{-1}\\
&\lesssim\sum_{|j-j_0|\leq2}|t|^{-\frac{1}{2}}+\sum_{j\leq j'_0}2^{2j}+|t|^{-1}\sum_{j\geq j'_0+1}2^{-2j}\\
&\lesssim|t|^{-\frac{1}{2}},
\end{split}
\end{equation*}
where $j'_0\in\mathbb{Z}$ is a fixed constant satisfying $2^{4j'_0}|t|\sim1$. 
\end{remark}

\begin{center}

\end{center}

\end{document}